\documentclass[12pt,reqno]{amsart}
\usepackage{amsmath,amsthm,amsfonts,amssymb}
\usepackage[foot]{amsaddr}
\usepackage{txfonts}
\usepackage{graphics}
\usepackage{color}
\usepackage{cite}
\usepackage{fleqn}
\usepackage{graphicx}
\usepackage{hyperref}
\usepackage{geometry}
\theoremstyle{plain}
\newtheorem{theorem}{Theorem}[section]
\newtheorem*{theorem*}{Theorem}
\newtheorem{lemma}{Lemma}[section]
\theoremstyle{remark}
\newtheorem{remark}{Remark}[section]

\theoremstyle{definition}
\newtheorem{definition}{Definition}[section]

\newcommand{\e}{_\varepsilon}
\newcommand{\eps}{{\varepsilon}}
\newcommand{\ds}{\displaystyle}
\renewcommand{\a}{\alpha}
\renewcommand{\b}{\beta}

\renewcommand{\d}{\mathrm{d}}
\newcommand{\cupl}{\bigcup\limits}
\newcommand{\suml}{\sum\limits}
\newcommand{\intl}{\int\limits}
\newcommand{\liml}{\lim\limits}

\newcommand{\A}{\mathcal{A}}
\renewcommand{\phi}{\varphi}

\begin{document}
\clearpage

\title[Gaps in the spectrum of a periodic quantum graph]
{Gaps in the spectrum of a periodic quantum graph with periodically distributed $\delta'$-type interactions}

\author[Diana Barseghyan]{Diana Barseghyan$^{1,2}$}
\address{$^1$ Department of Mathematics, University of 
Ostrava,  70103 Ostrava, Czech Republic}
\address{$^2$ Nuclear Physics Institute ASCR, 25068 \v{R}e\v{z} near Prague, Czech Republic}
\email{diana.barseghyan@osu.cz}

\author[Andrii Khrabustovskyi]{Andrii Khrabustovskyi$^3$}
\address{$^3$ Institute of Analysis, Karlsruhe Institute of Technology, 76133 Karlsruhe, Germany} 
\email{andrii.khrabustovskyi@kit.edu}

\begin{abstract}
We consider a family of quantum graphs 
$\{(\Gamma,\mathcal{A}_\varepsilon)\}_{\varepsilon>0}$, where 
$\Gamma$ is a $\mathbb{Z}^n$-periodic metric graph and the periodic Hamiltonian $\mathcal{A}_\varepsilon$ is defined by the operation $-\varepsilon^{-1} {\mathrm{d} ^2\over \mathrm{d} x^2}$ on the edges of $\Gamma$ and either $\delta'$-type conditions or the Kirchhoff conditions at its vertices. 
Here $\varepsilon>0$ is a small parameter. 
We show that the spectrum of $\mathcal{A}_\varepsilon$ 
has at least $m$ gaps as $\varepsilon\to 0$ ($m\in\mathbb{N}$ is a predefined number), 
moreover 
the location of these gaps can be nicely controlled via a suitable choice of 
the geometry of $\Gamma$ and of coupling constants 
involved in $\delta'$-type conditions. 
\\

\end{abstract}

\maketitle
\thispagestyle{empty}

\section{Introduction}

The name ``quantum graph'' is usually used for a pair $(\Gamma,\mathcal{A})$, 
where $\Gamma$ is a network-shaped 
structure of vertices connected by edges (``metric graph'') and $\mathcal{A}$ 
is a second order self-adjoint differential operator (``Hamiltonian''), which 
is determined by
differential operations on the edges and certain interface conditions at the 
vertices. Quantum graphs arise naturally in mathematics, physics, 
chemistry and engineering as models of wave propagation in 
quasi-one-dimensional systems looking like  
a narrow neighbourhood of a graph. One can mention, in particular, quantum wires, photonic
crystals, dynamical systems, scattering theory and many other 
applications. We refer to the recent monograph \cite{BK13} containing  a broad overview and comprehensive bibliography on this topic. 

In many applications (for instance, to graphen and carbon nano-structures -- cf. \cite{KP07,KL07}) 
periodic 
infinite graphs are studied. 
The metric graph $\Gamma$ is called  
\textit{periodic}
($\mathbb{Z}^n$-\textit{periodic}) if there is a group $G\simeq\mathbb{Z}^n$
acting isometrically, properly discontinuously and co-compactly on
$\Gamma$ (cf. \cite[Definition 4.1.1.]{BK13}). Roughly speaking $\Gamma$ is 
glued from countably many copies of a
certain compact graph $Y$ (``period cell") and each
$g\in G$ maps $Y$ to one of these copies.

In what follows in order to simplify the presentation (but without any loss of 
generality) 
we will assume that
$\Gamma$ is embedded into $\mathbb{R}^d$ with $d=3$ as $n=1,2$ and $d=n$ as 
$n\geq 3$ and is invariant under translations through linearly independent 
vectors $e_1,\dots,e_n$, i.e.
\begin{gather}\label{periodicity}
\Gamma=\Gamma+e_j,\ j=1,\dots,n.
\end{gather}
These vectors produce an action of $\mathbb{Z}^n$ on $\Gamma$.
Such an embedding can be always realized.

An example of $\mathbb{Z}^2$-periodic graph is presented on Figure \ref{fig1}, 
its period cell is highlighted in bold lines.
\begin{figure}[h]\large
\begin{center}
\scalebox{0.7}{ \includegraphics{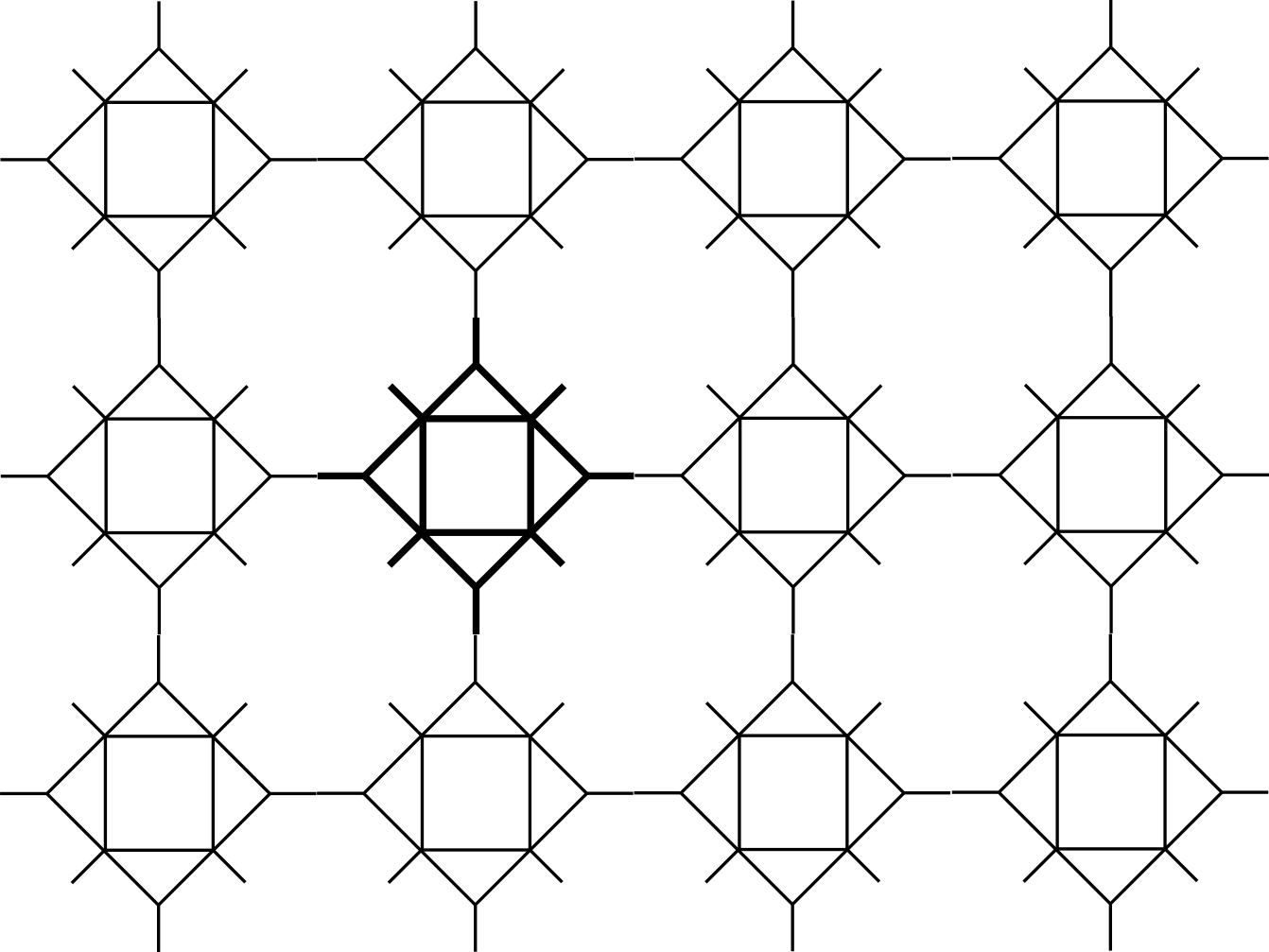}}
\caption{An example of $\mathbb{Z}^2$-periodic graph}\label{fig1}
\end{center}
\end{figure}

The Hamiltonian $\mathcal{A}$ on a periodic metric graph $\Gamma$ is said to be 
periodic if it commutes with the action of $\mathbb{Z}^n$ on $\Gamma$.
It is well-known (see, e.g., \cite[Chapter 4]{BK13}) that the
spectrum of such operators has a
band structure, i.e. it is a locally finite union of compact intervals called
\textit{bands}. In general the neighbouring bands may overlap. The bounded open 
interval $(\a,\b)\subset\mathbb{R}$ is called a \textit{gap} if it has an 
empty 
intersection with the spectrum, but its ends belong to it. In general the 
presence of
gaps in the spectrum is not guaranteed -- for example if $\Gamma$ is a
rectangular lattice and $\mathcal{A}$ is defined by the operation  
$-{\d ^2/ \d x^2}$ on its edges and the Kirchhoff 
conditions at the vertices  
then the spectrum $\sigma(\mathcal{A})$ of the operator $\mathcal{A}$ has no 
gaps, namely $\sigma(\mathcal{A})=[0,\infty)$.
Existence of spectral gaps is important because of various applications,
for example in physics of photonic crystals.

There are several ways how to create quantum waveguides with spectral gaps.
One of them is to use decorating graphs. Namely, given a fixed graph $\Gamma_0$ 
we ``decorate'' it attaching to each vertex of $\Gamma_0$ a copy of certain 
fixed graph $\Gamma_1$, the obtained graph we denote by $\Gamma$. 
Spectral properties of such graphs were studied in \cite{SA00},  
where operators defined on functions 
on vertices were considered (``discrete graphs''). The case  of quantum graphs was 
studied in 
\cite{Ku05} and it was proved that gaps open up in the spectrum of the operator 
defined by the operation $-{\d ^2/ \d x^2}$ on the edges of $\Gamma$ and the 
Kirchhoff conditions at the vertices (other conditions are also 
allowed); these gaps are located around eigenvalues of a certain Hamiltonian on 
$\Gamma_1$. 

Also one can use  ``spider decoration'' procedure:
in each vertex we disconnect the edges emerging from it
and then connect their loose endpoints by a certain additional graph (``spider''). 
Such decorating procedure was probably used  for 
the first time in \cite{AEL94,Ex95}, more results  on gap opening one can find 
in \cite{O06}.

Another way to create gaps, instead of to perturb a graph geometry, is to 
substitute the Kirchhoff conditions at the vertices by more ``advanced'' ones. 
For example, let $\Gamma$ be a
rectangular lattice and $\mathcal{A}$ be defined by the operation  
$-{\d ^2/ \d x^2}$ on its edges and $\delta$ conditions at the vertices, i.e. 
the functions from $\mathrm{dom}(\mathcal{A})$ are continuous 
at all vertices and the sum of their derivatives is proportional 
to the value of a function at the vertex with a coupling constant 
$\a\in\mathbb{R}$ (the case $\a=0$ corresponds to the Kirchhoff 
conditions). Then (cf. \cite{Ex95,Ex96}) the spectrum $\sigma(\mathcal{A})$ 
has infinitely many gaps provided $\a\not=0$ and the lattice-spacing ratio satisfies some additional 
mild assumptions. 

The goal of the current paper is to study spectral properties of some specific class of periodic quantum graphs. The main peculiarity of these graphs 
is that their spectral gaps can be nicely controlled via a suitable choice  of  the graph geometry and 
of coupling constants involved in interface conditions at its
vertices. 

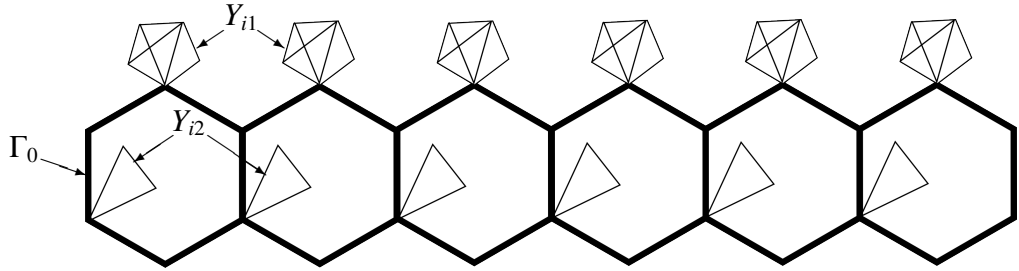
\begin{figure}[h]
\begin{center}
\begin{picture}(350,100)
\scalebox{0.15}{\includegraphics{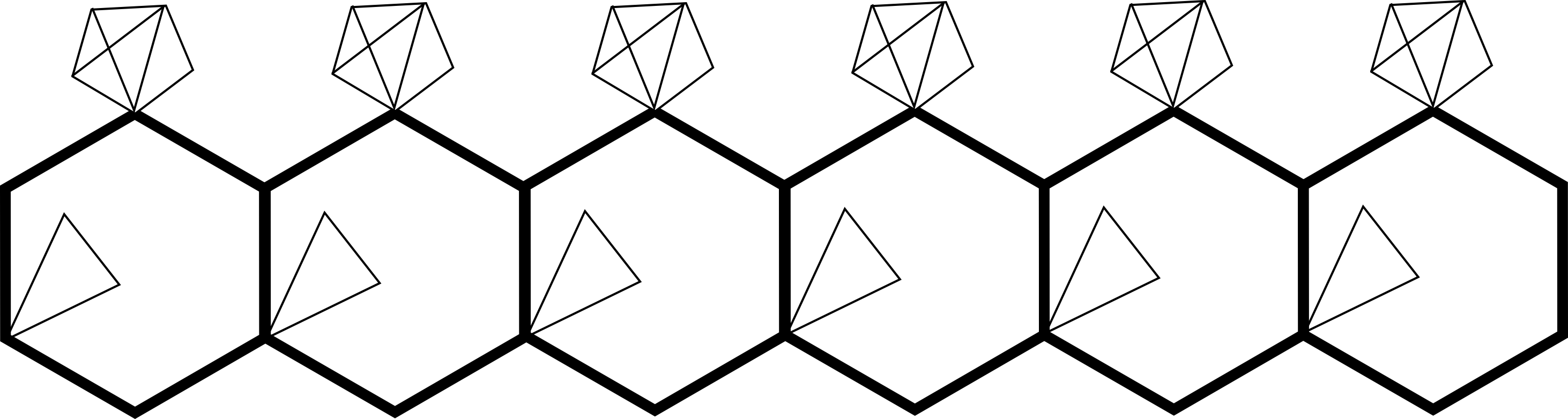}}

\put(-300,92){$Y_{i1}$}
\put(-320,52){$Y_{i2}$}
\put(-381,42){$\Gamma_0$}

\put(-300,90){\vector(-3,-2){10}}
\put(-287,90){\vector( 3,-2){10}}

\put(-319,50){\vector(-3,-2){15}}
\put(-306,50){\vector( 3,-2){22}}

\put(-369,41){\vector( 3,-1){17}}

\end{picture}\caption{The example of the graph $\Gamma$. Here 
$m=2$.}\label{fig2}
\end{center}
\end{figure}

In particular, for a given $m\in\mathbb{N}$ we construct a family 
$\left\{(\Gamma,\mathcal{A}\e)\right\}_{\eps>0}$ of periodic quantum graphs 
having at least $m$ gaps as $\eps$ is small enough and moreover the first $m$
gaps converge to predefined intervals as $\eps\to 0$. The graph $\Gamma$ is 
constructed as follows. We take an arbitrary $\mathbb{Z}^n$-periodic graph 
$\Gamma_0$  with vectors $e_1,\dots,e_n$ producing an action of $\mathbb{Z}^n$ 
on it and attach to $\Gamma_0$ a family of compact graphs 
$Y_{ij}$, $i=(i_1,\dots,i_n)\in\mathbb{Z}^n$, $j=1,\dots,m$ satisfying
$Y_{0j}+\suml_{k=1}^n i_k e_k=Y_{ij}$. We denote by  $\Gamma$ the obtained graph (an example is presented on Figure \ref{fig2})
and consider on it the Hamiltonian $\mathcal{A}\e$ defined by the operation 
$$-\ds{\eps^{-1}}{\d ^2\over \d x^2}$$ on its edges and the Kirchhoff 
conditions in 
all its vertices except the points of attachment of $Y_{ij}$ to $\Gamma_0$ -- in 
these points we pose $\delta'$-type conditions 
(in the case of vertex with two outgoing edges they coincide with the usual 
$\delta'$ conditions at a point on the line -- cf. \cite[Sec. I.4]{AGHH05}).
The required structure for the spectrum of $\mathcal{A}\e$ is achieved via a suitable choice of 
coupling constants involved in $\delta'$-type conditions  and of "sizes" of attached graphs.

\section{\label{sec1}Setting of the problem and main result}

\subsection{\label{subsec11}Graph $\Gamma$}
 
Let $$\Gamma=(\mathcal{V},\mathcal{E},\gamma,l)$$ be a connected 
$\mathbb{Z}^n$-periodic metric graph. Here 
\begin{itemize}
\item[-] by $\mathcal{V}$ we denote the set of its vertices,

\item[-] by $\mathcal{E}$ we denote the set of its edges, 

\item[-] the map $\gamma: \mathcal{E}\to \mathcal{V}\times \mathcal{V}$ assigns 
to each edge $e$ its initial and terminal points (we denote them $\gamma^-(e)$ 
and $\gamma^+(e)$, correspondingly),  

\item[-] the function $l:\mathcal{E}\to (0,\infty)$ 
assigns to the edge $e$ its length $l(e)$. 

\end{itemize}
We suppose that the degree of each vertex (i.e., the number of edges emanating from it) is finite.
In order to simplify the presentation we assume that $\Gamma$ is embedded into $\mathbb{R}^d$, where $d=3$ as 
$n=1,2$ and $d=n$ as 
$n\geq 3$.

On each edge $e\in\mathcal{E}$ we introduce the local coordinate $x_e\in 
[0,l(e)]$ in such a way that $x_e=0$ corresponds to $\gamma^-(e)$ and 
$x_e=l(e)$ 
corresponds to $\gamma^+(e)$.
One can assume that $\Gamma$ has no loops (i.e. there 
is no edge $e$ with $\gamma^+(e)=\gamma^-(e)$), otherwise one can break them 
into pieces by introducing a new intermediate vertex.
For $v\in\mathcal{V}$ we denote 
$$\mathcal{E}(v)=\left\{\text{the set of edges outgoing from }v\right\}.$$

In a natural way the function $l$ gives rise to a metric on $\Gamma$. In what 
follows by $\overset{\circ}{Z}$ (or $\mathrm{int}Z$), $\overline{Z}$, $\partial 
Z$ we denote, correspondingly, the interior, the closure, the boundary of a 
subset $Z\subset\Gamma$ with respect to this metric.
In particular, $\partial\Gamma$ consists of the vertices of $\Gamma$ of degree $1$.

The $\mathbb{Z}^n$-periodicity of $\Gamma$ means that there are linearly 
independent vectors $e_1,\dots,e_n$ satisfying \eqref{periodicity}.
By $Y$ we denote a \textit{period cell} of $\Gamma$, that is a compact subset of 
$\Gamma$
satisfying 
\begin{gather*}
\overset{\circ}{Y}\cap \left(\overset{\circ}{Y}+ \suml_{k=1}^n i_k e_k\right) 
=\varnothing\text{ for an arbitrary }i=(i_1,\dots,i_n)\in 
\mathbb{Z}^n\setminus\{0\},\\
\Gamma=\cupl_{i\in\mathbb{Z}^n} \left(Y+\suml_{k=1}^n i_k e_k\right).
\end{gather*}
We notice that period cell is not uniquely defined.  

The period cell $Y$ can be always chosen in such a way that $\partial Y$ does 
not contain any vertex $v\in\mathcal{V}\setminus\partial\Gamma$ (see Figure 
\ref{fig1}).
Under such a choice of the period cell  
the boundary $\partial Y$ of $Y$ consists of two disjoint parts $\partial_{\mathrm{int}}Y$ and $\partial_{\mathrm{ext}}Y$, where
\begin{itemize}
\item $\partial_{\mathrm{int}}Y$ consists of vertices of $\Gamma$ of degree $1$ belonging to $Y$,

\item $\partial_{\mathrm{ext}}Y$ consists of vertices of $Y$ of degree $1$ lying in the interiors of certain edges of $\Gamma$.

\end{itemize}
An example of $\mathbb{Z}^2$-periodic graph is presented on Figure \ref{fig1}. 
Its period cell $Y$ is presented in more details on Figure \ref{fig3} and 
one has
$$\partial_{\mathrm{int}}Y=\left\{v_{13},v_{14},v_{15},v_{16},\right\},\quad \partial_{\mathrm{ext}}Y=\left\{v_1,v_5,v_8,v_{11}\right\}.$$

\begin{figure}[h]\large
\begin{center}
\begin{picture}(225,230)      
\scalebox{0.7}{ \includegraphics{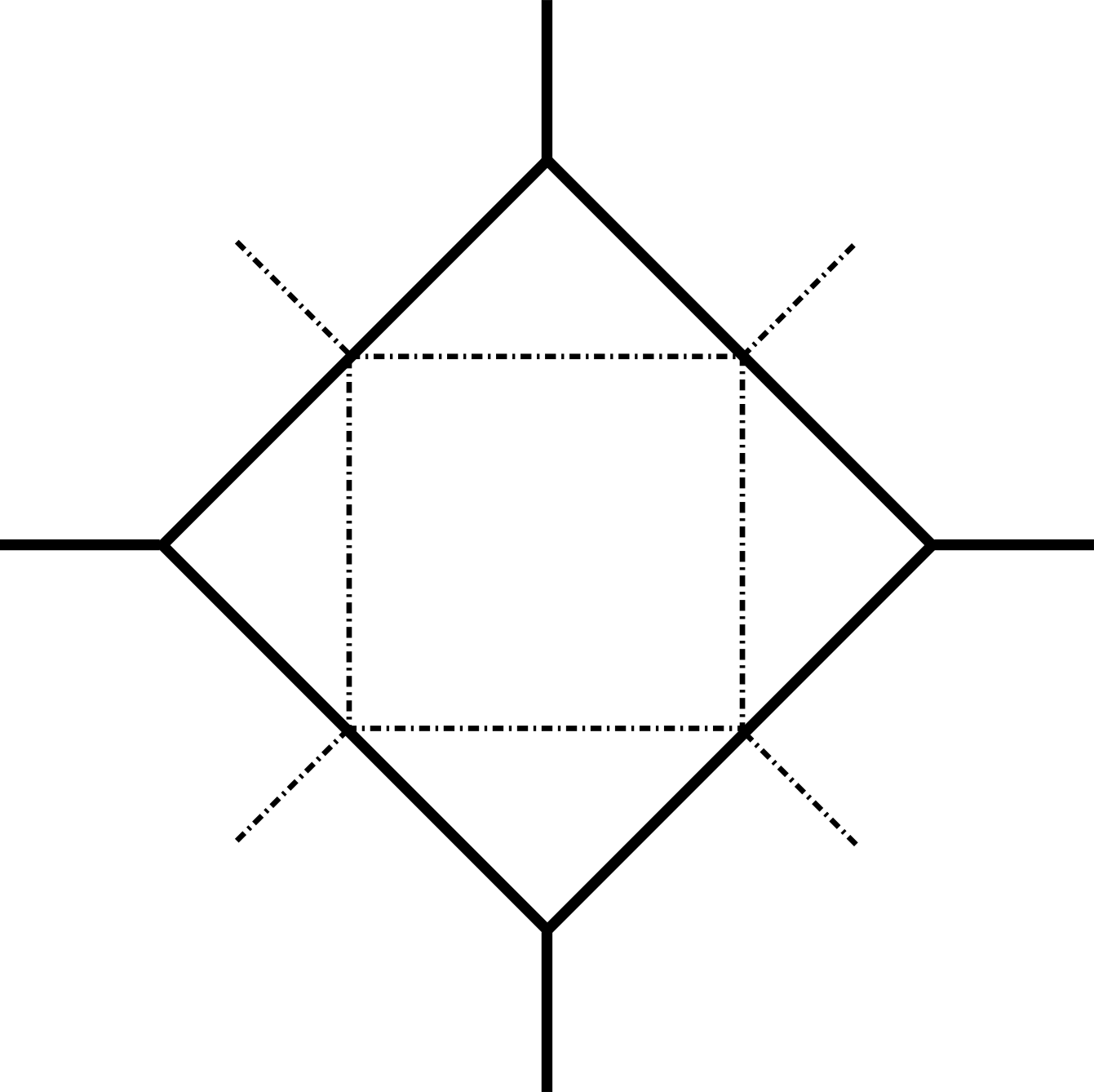}}

\put(-240,111){$v_{1}$}
\put(-185,111){$v_{2}$}
\put(-150,141){$v_{3}$}
\put(-120,175){$v_{4}$}
\put(-120,230){$v_{5}$}
\put(-87,141){$v_{6}$}
\put(-54,111){$v_{7}$}
\put(2,111){$v_{8}$}
\put(-88,81){$v_{9}$}
\put(-150,81){$v_{12}$}
\put(-120,45){$v_{10}$}
\put(-120,-10){$v_{11}$}
\put(-195,42){$v_{13}$}
\put(-195,180){$v_{14}$}
\put(-47,180){$v_{15}$}
\put(-47,42){$v_{16}$}

\put(-215,110){$^{e_{1}}$}
\put(-180,131){$^{e_{2}}$}
\put(-173,154){$^{e_{3}}$}
\put(-145,166){$^{e_{4}}$}
\put(-123,200){$^{e_{5}}$}
\put(-90,166){$^{e_{6}}$}
\put(-60,154){$^{e_{7}}$}
\put(-55,131){$^{e_{8}}$}
\put(-25,111){$^{e_{9}}$}
\put(-55,80){$^{e_{10}}$}
\put(-60,57){$^{e_{11}}$}
\put(-90,45){$^{e_{12}}$}
\put(-126,10){$^{e_{13}}$}
\put(-145,45){$^{e_{14}}$}
\put(-178,59){$^{e_{15}}$}
\put(-182,81){$^{e_{16}}$}
\put(-166,105){$^{e_{17}}$}
\put(-118,150){$^{e_{18}}$}
\put(-71,105){$^{e_{19}}$}
\put(-118,62){$^{e_{20}}$}

\end{picture}
\end{center}
\caption{\label{fig3}Period cell of the graph from Figure \ref{fig1}}
\end{figure}
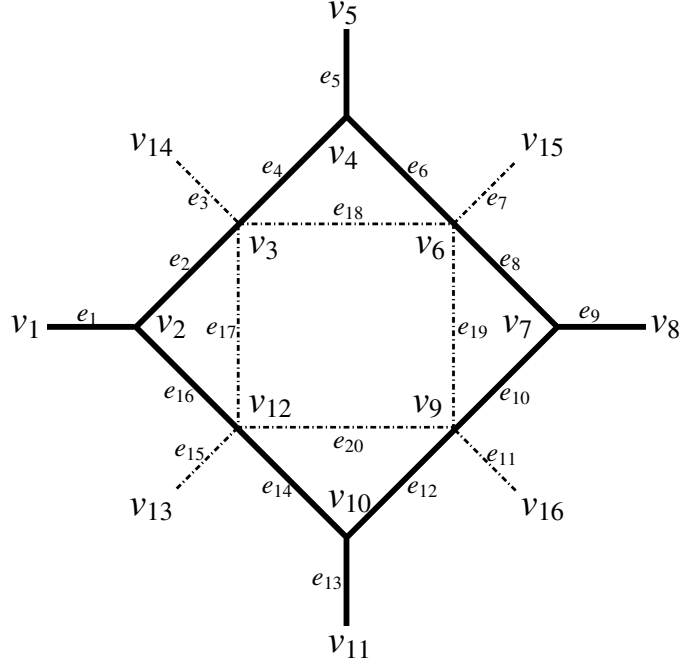

Additionally, we suppose that $Y$ can be expressed as a union of $m+1$ compact
subsets, 
\begin{gather}\label{decomp}
Y=\cupl_{j=0}^m Y_j,\ m\in\mathbb{N},
\end{gather}
satisfying the following conditions:
\begin{gather}\label{conditions}
\begin{array}{ll}
\text{(i)\quad }&\overset{\circ}{Y_{j}}\not=\varnothing,\\\vspace{1pt}
\text{(ii)\quad }&Y_j\text{ are connected,}\\\vspace{1pt}\text{(iii)\quad }&
\textrm{int}(Y_j\cap  Y_{k})=\varnothing\text{ provided 
}j\not=k,\\ &\text{i.e. $Y_j$ and $Y_k$ may have only common vertices, not edges},\\\vspace{1pt}\text{(iv)\quad }&
\partial_{\mathrm{ext}}Y\subset \partial Y_0,\\\vspace{1pt}\text{(v)\quad }&
\text{the sets }
\mathcal{V}_j:=\partial Y_0\cap \partial Y_j,\ j=1,\dots, m\text{
are nonempty},\\\vspace{1pt}\text{(vi)\quad }&
\text{if }j,k\not=0\text{ and }j\not=k\text{ then either }\partial Y_j\cap\partial 
Y_k=\varnothing\text{ or }
\partial Y_j\cap\partial Y_k \subset \partial 
Y_0.
\end{array}
\end{gather}

\begin{remark}
In fact, decomposition \eqref{decomp} satisfying \eqref{conditions} is always possible for an arbitrary graph $\Gamma\ncong\mathbb{R}$ under a suitable choice of a period cell. Let us formulate this statement more accurately. 
At first we notice that for an arbitrary $s\in\mathbb{N}$ condition \eqref{periodicity} holds with $e^s_j:=s e_j$ instead of $e_j$, $j=1,\dots,n$ 
(i.e., $\Gamma$ is periodic with respect to the period cell $Y^s:=sY$). It is easy to see that if $\Gamma\ncong\mathbb{R}$ then $Y^s$ contains $m$ edges  $\tilde e_1,\dots,\tilde e_m$ satisfying
$$\tilde e_i\cap \tilde e_j=\varnothing\text{ as }i\not=j,\quad Y_0:=\overline{Y^s\setminus \cupl_{j=1}^n {\tilde e_j}}\text{ is a connected set},\quad \overset{\circ}{Y_{0}}\not=\varnothing,\quad \partial_{\mathrm{ext}}Y^s\subset \partial Y_0$$
provided $s$ is large enough. We set $Y_j:=\tilde e_j$, $j=1,\dots,m$.
Obviously, $Y^s=\cupl_{j=0}^m Y_j$ and conditions \eqref{conditions} hold true.
\end{remark}

It is easy to see that $\cupl_{j=1}^m\mathcal{V}_j\subset\overset{\circ}Y$.
One can  assume that the set 
$\cupl_{j=1}^m\mathcal{V}_j$ belongs to $\mathcal{V}$, otherwise if some of its 
points belongs to the interior of an edge then we can add it to $\mathcal{V}$  (as 
a vertex with two outgoing edges). 
Finally, for $i=(i_1,\dots,i_n)\in\mathbb{Z}^n$, $j\in\{1,\dots,m\}$ we set
$$\mathcal{V}_{ij}= \mathcal{V}_j+ \suml_{k=1}^n i_k e_k.$$
The points belonging to $\mathcal{V}_{ij}$ will support our $\delta'$-type
conditions. 

Also we will use the notation
$$Y_{ij}:=Y_j+\suml_{k=1}^n i_k e_k,\quad i=(i_1,\dots,i_n)\in\mathbb{Z}^n,\ j\in\{1,\dots,m\}.$$

Let us come back to the example depicted on Figure \ref{fig3}. There are 
several possibilities to decompose the period cell in a way described above. 
For 
example, 
one has
$Y=Y_0\cup Y_1$,  
\begin{itemize}
\item[] $Y_0$ consists of the edges $e_1,e_2,e_4,e_5,e_6,e_8,e_9,e_{10},e_{12},e_{13},e_{14},e_{16}$ (solid 
lines),

\item[] $Y_1$ consists of the edges $e_3,e_7,e_{11},e_{15},e_{17},e_{18},e_{19},e_{20}$ 
(dash-dotted lines).
\end{itemize}
The set $\mathcal{V}_1$ consists of the vertices $v_3,v_6,v_9,v_{12}$.

One can also decompose $Y$ in a more ``advanced'' way, for example as a union 
of 
six sets:
\begin{itemize}
\item[] $Y_0$ consists of the edges $e_1,e_2,e_4,e_5,e_6,e_8,e_9,e_{10},e_{12},e_{13},e_{14},e_{16}$,

\item[] $Y_1$ consists of the edges $e_{17},e_{18},e_{19},e_{20}$, 

\item[] $Y_2$ consists of the edge $e_3$,

\item[] $Y_3$ consists of the edge $e_7$,

\item[] $Y_4$ consists of the edge $e_{11}$, 

\item[] $Y_5$ consists of the edge $e_{15}$.

\end{itemize}
Then
$\mathcal{V}_1=\{v_3,v_6,v_9,v_{12}\}$,
$\mathcal{V}_2=\{v_3\},\ \mathcal{V}_3=\{v_6\},\ \mathcal{V}_4=\{v_9\},\ 
\mathcal{V}_5=\{v_{12}\}.$

\subsection{\label{subsec12}Hamiltonian $\mathcal{A}\e$}

In what follows if $u:\Gamma\to\mathbb{C}$ and 
$e\in\mathcal{E}$ then by $u_e$ we denote
the restriction of $u$ onto $\overset{\circ}e$. Via a local coordinate 
$x_e$ we identify $u_e$ with a function on $(0,l(e))$.  

We introduce several functional spaces on $\Gamma$.
The space $L_2(\Gamma)$ consists of functions that are measurable and  
square integrable on each edge and such that 
$$\|u\|^2_{L_2(\Gamma)}:=\suml_{e\in 
\mathcal{E}}\|u_e\|^2_{L_2\left(0,l(e)\right)}=\suml_{e\in 
\mathcal{E}}\intl_0^{l(e)}|u_e(x_e)|^2 \d x_e<\infty.$$

The space $\widetilde H^k(\Gamma)$, $k\in\mathbb{N}$ consists of functions on 
$\Gamma$  
belonging to the Sobolev space $H^k(e)$ on each edge $e\in\mathcal{E}$ and 
satisfying
 $$\|u\|^2_{\widetilde H^k(\Gamma)}:=\suml_{e\in 
\mathcal{E}}\|u_e\|^2_{H^k\left(0,l(e)\right)}=\suml_{e\in 
\mathcal{E}}\suml_{l=0}^k\intl_0^{l(e)}\left|\d^l u_e(x_e)\over \d 
x_e^l\right|^2 \d 
x_e<\infty.$$

Finally, the set $\mathcal{H}(\Gamma)$ consists of functions $u\in \widetilde 
H^1(\Gamma)$ satisfying the following conditions at vertices of $\Gamma$:
\begin{itemize}
\item if $v\in \mathcal{V}\setminus\left( \cupl_{i\in\mathbb{Z}^n}\cupl_{j=1}^m 
\mathcal{V}_{ij}\right)$ then $u$ is continuous at $v$, i.e. the limiting value 
of $u(x)$ when $x$ approaches $v$ along $e\in\mathcal{E}(v)$ is independent of 
$e$. We denote this value by $u(v)$;

\item if $v\in\mathcal{V}_{ij}$ for some $i=(i_1,\dots,i_n)\in\mathbb{Z}^n$, $j\in\{1,\dots,m\}$ 
 then  
\begin{itemize}

\item the limiting value of $u(x)$ when $x$ approaches $v$ along 
$e\in\mathcal{E}(v)\cap Y_{i0}$ is independent of $e$. We denote this value by 
$u_{0}(v)$;

\item the limiting value of $u(x)$ when $x$ approaches $v$ along 
$e\in\mathcal{E}(v)\cap Y_{ij}$ is independent of $e$. We denote this value by 
$u_{j}(v)$.

\end{itemize}

\end{itemize}

Now, we describe the family of operators $\mathcal{A}\e$, which will the main 
object of our
interest in this paper. 
In $L_2(\Gamma)$ we introduce the sesquilinear form $a\e$,
\begin{gather}
\notag
\mathrm{dom}(a\e)=\mathcal{H}(\Gamma),\\
\label{a}
a\e[u,w]={\eps^{-1}}\suml_{e\in\mathcal{E}}\intl_0^{l(e)} {\d u_e\over \d 
x_e}\overline{ \d w_e\over \d x_e} \d x_e
+\suml_{i\in \mathbb{Z}^n}\suml_{j=1}^m \suml_{v\in \mathcal{V}_{ij}}
q_j \left(u_{0}(v)-u_{j}(v)\right)\overline{\left(w_{0}(v)-w_{j}(v) \right)},
\end{gather}
where $q_j$  are positive constants.
The definition of $a\e[u,w]$ makes sense: the second term in the right-hand-side of \eqref{a} (we 
denote it $\tilde a[u,w]$) is finite, namely one has the estimate 
\begin{gather*}
|\tilde a[u,w]|^2\leq C \|u\|^2_{\widetilde H^1(\Gamma)} \|w\|^2_{\widetilde  
H^1(\Gamma)}
\end{gather*}
following from the standard trace inequality
$$|u(l)|^2\leq 2\left({l}^{-1}\|u\|^2_{L_2(0,l)}+l \|u'\|^2_{L_2(0,l)}\right),\ 
\forall u\in H^1(0,l).$$
Furthermore, it is straightforward to check that the form $a\e[u,v]$ is symmetric, densely defined, closed and positive. Then (see, e.g., \cite[Theorem VIII.15]{RS78})  
there exists the unique self-adjoint and positive operator $\mathcal{A}\e$ 
associated with the form $a\e$, i.e.
\begin{gather*}
(\mathcal{A}\e u,v)_{L_2(\Gamma)}= a\e[u,v],\quad\forall u\in
\mathrm{dom}(\mathcal{A}\e),\ \forall  v\in \mathrm{dom}(a\e).
\end{gather*}

The definitional domain of the operator $\mathcal{A}\e$ consists of functions 
$u\in \mathcal{H}(\Gamma)$ belonging to 
$\widetilde H^2(\Gamma)$
and satisfying the following conditions at the vertices 
(additionally to the conditions needed to be in $\mathcal{H}(\Gamma)$):
\begin{itemize}
\item if $v\in \mathcal{V}\setminus \cupl_{i\in\mathbb{Z}^n}\cupl_{j=1}^m 
\{v_{ij}\}$ then
\begin{gather*}
\suml_{e\in\mathcal{E}(v)} \left.{\d u_e\over \d 
\mathbf{x}_e}\right|_{\mathbf{x}_e=0}=0\qquad (\text{Kirchhoff conditions}),
\end{gather*}
where 
$$\mathbf{x}_e=\begin{cases}x_e,&\text{if 
}v=\gamma^-(e),\\\mathbf{x}_e=l(e)-x_e,&\text{if } v=\gamma^+(e)\end{cases}$$ 
(i.e. $\mathbf{x}_e$ is a natural coordinate on $e\in\mathcal{E}(v)$ such that $\mathbf{x}_e=0$ at $v$);

\item if $v\in\mathcal{V}_{ij}$, $i\in\mathbb{Z}^n$, $j\in\{1,\dots,m\}$ one has 
the following conditions at $v$:
\begin{gather}\label{delta'}
\begin{array}{c}
\ds-\eps^{-1}\suml_{e\in \mathcal{E}(v)\cap Y_{i0}}\left.{\d u_e\over \d 
\mathbf{x}_e}\right|_{\mathbf{x}_e=0}+q_{j}(u_0(v)-u_j(v))=0,\\
\ds-\eps^{-1}\suml_{e\in \mathcal{E}(v)\cap Y_{ij}}\left.{\d u_e\over \d 
\mathbf{x}_e}\right|_{\mathbf{x}_e=0}+q_{j}(u_j(v)-u_0(v))=0.
\end{array}\qquad\text{($\delta'$-type conditions)}
\end{gather}
\end{itemize}
The operator $\mathcal{A}\e$ acts as follows:
$$\left(\mathcal{A}\e u\right)_e= -\eps^{-1}{\d ^2 u\e\over \d x^2_e}  ,\ e\in\mathcal{E}.$$

\begin{remark}
Suppose that $v\in\mathcal{V}_{ij}$ (for some 
$i\in\mathbb{Z}^n$, $j\in\{1,\dots,m\}$) has two outgoing edges, 
$e\in\mathcal{E}(v)\cap Y_{ij}$ and $\tilde e\in\mathcal{E}(v)\cap Y_{i0}$.
Then, evidently, conditions \eqref{delta'} are equivalent to 
\begin{gather*}
\left.{\d u_e\over \d 
\mathbf{x}_e}\right|_{\mathbf{x}_e=0}+\left.{\d u_{\tilde e}\over \d 
\mathbf{x}_{\tilde e}}\right|_{\mathbf{x}_{\tilde e}=0}=0,\quad 
\kappa\ds\left.{\d u_e\over \d 
\mathbf{x}_e}\right|_{\mathbf{x}_e=0}=\left(\left. 
u_e\right|_{\mathbf{x}_e=0}-\left.u_{ \tilde e}\right|_{\mathbf{x}_{\tilde 
e}=0}\right),\ \kappa=(q_{j}\eps)^{-1},
\end{gather*}
(recall that $\mathbf{x}_e\in [0,l(e)]$ and $\mathbf{x}_{\tilde e}\in 
[0,l(\tilde e)]$ are the natural coordinates on $e$ and $\tilde e$, 
correspondingly, such that $\mathbf{x}_e=\mathbf{x}_{\tilde e}=0$ at $v$).
Thus we obtain the usual $\delta'$ conditions at a point on the line  
\cite[Sec. I.4]{AGHH05} that explains why we use the term ``$\delta'$-type 
conditions'' for \eqref{delta'}.
Various analogues of $\delta'$ conditions for graphs are discussed 
in \cite{Ex96}.
\end{remark}

\begin{remark}
The name ``$\delta'$-conditions'' is  misleading because such Hamiltonians
cannot be obtained using families of scaled zero-mean potentials (cf. \cite{S86}). 
However one can approximate them by a triple of $\delta$ potentials and then by regular $\delta$-like ones following an idea put forward in \cite{CS98} and then made mathematically rigorous in \cite{ENZ01}. The problem of approximating all singular vertex couplings (in particular, $\delta'$-type ones) in a quantum graph is solved in \cite{CET10}.
\end{remark}

\subsection{\label{subsec13}The main result}

Before to formulate the result let us introduce several notations. 

We denote 
\begin{itemize}
\item
by $l_j$, $j=0,\dots,m$  the total length of all edges belonging to $Y_j$,

\item 
by $N_j$, $j=1,\dots,m$ we denote the number of points belonging to the set $\mathcal{V}_j$.
\end{itemize}

Then for $j=1,\dots,m$ we set:
\begin{gather}\label{aj}
a_j:={N_j q_j\over l_j}.
\end{gather}
It is assumed that the numbers $a_j$ are pairwise non-equivalent. We renumber them in the 
ascending order, that is
\begin{gather}\label{aj_ord}
a_j<a_{j+1},\ \forall j=1,\dots, m -1.
\end{gather}

Finally, we consider the following equation (with unknown
$\lambda\in\mathbb{C}$):
\begin{gather}\label{mu_eq}
\mathcal{F}(\lambda):=1+\suml_{i=1}^m{a_i l_i\over l_0(a_i-\lambda)}=0.
\end{gather}
It is straightforward to show that if
\eqref{aj_ord} holds then equation \eqref{mu_eq} has exactly
$m$ roots, they are real and interlace with $a_j$. We denote
them by $b_j$, $j=1,\dots,m$ supposing that they are renumbered in
the ascending order, i.e.
\begin{gather}\label{abab}
a_j<b_j<a_{j+1},\ j={1,\dots,m-1},\quad
a_m<b_m<\infty.
\end{gather}

We are now in position to formulate the first main result of this work.

\begin{theorem}\label{th1}
Let $L>0$ be an arbitrary number. Then the spectrum of the operator 
$\mathcal{A}\e$ in $[0,L]$ has the following structure for $\eps$ small enough:
\begin{gather}\label{th1_f1}
\sigma(\mathcal{A}\e)\cap [0,L]=[0,L]\setminus
\cupl_{j=1}^m\big(a_j(\eps),b_j(\eps)\big),
\end{gather}
where the endpoints of the intervals $\big(a_j(\eps),b_j(\eps)\big)$ satisfy the 
relations
\begin{gather}\label{th1_f2}
\liml_{\eps\to 0}a_j(\eps)=a_j,\quad \liml_{\eps\to 0}b_j(\eps)=b_j,\ 
j=1,\dots,m.
\end{gather}
\end{theorem}

In the last section we will present our second result (Theorem \ref{th2}): we will show that under a suitable choice of the graph $\Gamma$ and the coupling constants $q_j$ the limit intervals $(a_j,b_j)$ coincide with predefined ones.

Theorem \ref{th1} will be proven in the next section. We postpone the outline of the 
proof to the end of Subsection \ref{subsec21} because we need to introduce first 
some more notations.

\section{\label{sec2} Proof of Theorem \ref{th1}}
\subsection{\label{subsec21} Preliminaries}

The Floquet-Bloch theory establishes a relationship between the spectrum of 
$\A\e$ and the spectra of appropriate operators on $Y$. Namely, let $$\theta\in 
\mathbb{T}^n=\{\theta=(\theta_1,\dots,\theta_n)\in\mathbb{C}^n,\ 
|\theta_k|=1\text{ for all }k=1,\dots,n\}.$$ 
We denote by $\mathcal{H}^\theta(\Gamma)$ the set of functions $u:\Gamma\to\mathbb{C}$ satisfying
\begin{itemize}
\item $\forall e\in\mathcal{E}$: $u_e\in H^1(e)$,

\item $u$ is continuous at all vertices belonging to $\mathcal{V}\setminus\left(\cupl_{i\in\mathbb{Z}^n}\cupl_{j=1}^m\mathcal{V}_{ij}\right)$,

\item at the vertices belonging to $\cupl_{i\in\mathbb{Z}^n}\cupl_{j=1}^m\mathcal{V}_{ij}$
$u$ satisfies the same conditions as functions from $\mathcal{H}(\Gamma)$,

\item $u$ is $\theta$-periodic, that is
$$\forall x\in\Gamma: u(x+e_k)=\theta_k u(x),\ k=1,\dots,n$$
(if $\theta=(1,1,\dots,1)$ (respectively, $\theta=-(1,1,\dots,1)$) one has periodic (respectively, antiperiodic) conditions).

\end{itemize}
Then we introduce the 
sesquilinear form $a\e^\theta$ defined as follows (below the notation $\mathcal{E}(Y)$ stays for the set of edges of $Y$): 
\begin{gather*}
\mathrm{dom}(a\e^\theta)=\left\{u=v|_Y,\ v\in\mathcal{H}^\theta(\Gamma)\right\},\\
a^\theta\e[u,w]=\eps^{-1}\suml_{e\in\mathcal{E}(Y)}\intl_0^{l(e)} {\d 
u_e\over \d x_e}\overline{ \d w_e\over \d x_e} \d x_e
+\suml_{j=1}^m \suml_{v\in \mathcal{V}_{j}}
q_j \left(u_{0}(v)-u_{j}(v)\right)\overline{\left(w_{0}(v)-w_{j}(v) \right)}.
\end{gather*}
We define $\A^\theta_{\eps}$ as the operator acting in $L_{2}(Y)$ being 
associated with the form $a^\theta_{\eps}$.
Since $Y$ is compact, $\A^\theta_{\eps}$ has a purely discrete 
spectrum. We denote by $\left\{\lambda^{\theta}_{k} 
(\eps)\right\}_{k\in\mathbb{N}}$ the sequence of eigenvalues of 
$\A^\theta_{\eps}$ arranged in the ascending order and repeated according to 
their multiplicity.

One has the following representation (see \cite[Chapter 4]{BK13}):
\begin{gather}\label{repres1}
\sigma(\A\e)=\cupl_{k=1}^\infty \cupl_{\theta\in \mathbb{T}^n}
\left\{\lambda^{\theta}_{k}(\eps)\right\}.
\end{gather}
Moreover, for any fixed $k\in\mathbb{N}$ the set 
\begin{gather}\label{repres1+}
L_k(\eps):=\cupl_{\theta\in \mathbb{T}^n} 
\left\{\lambda^{\theta}_{k}(\eps)\right\}
\end{gather}
is a compact interval ($k$-th spectral band). 

By $\mathcal{H}(Y)$ we denote the set of functions $u:Y\to\mathbb{C}$ satisfying
\begin{itemize}

\item $\forall e\in\mathcal{E}(Y)$: $u_e\in H^1(e)$,

\item $u$ is continuous at all vertices 
of $Y$ except those ones belonging to $\cupl_{j=1}^m\mathcal{V}_{j}$,

\item at the vertices from $\cupl_{j=1}^m\mathcal{V}_{j}$\hspace{2mm}
$u$ satisfies the same conditions as functions from $\mathcal{H}(\Gamma)$. 
\end{itemize}
Then we introduce the operator $\mathcal{A}\e^N$ (respectively, 
$\mathcal{A}\e^D$) 
as the operator acting in $L_2(Y)$ and  associated with the sesquilinear form
$a\e^{N}$  (respectively, $a\e^D$) defined as follows:
\begin{gather*}
\mathrm{dom}(a\e^N)=\mathcal{H}(Y),\ a\e^N[u,w]=a\e^\theta[u,w],
\\
\text{(respectively, }\mathrm{dom}(a\e^D)=\{u\in \mathcal{H}(Y):\ u=0\text{ on 
}\partial_{\mathrm{ext}}Y\},\ a\e^D[u,w]=a\e^\theta[u,w]\text{).}
\end{gather*}
The subscript $N$ (respectively, $D$) indicates that functions from 
$\mathrm{dom}(\mathcal{A}\e^N)$ (respectively, $\mathrm{dom}(\mathcal{A}\e^D)$) 
satisfy the Neumann (respectively, Dirichlet) boundary conditions on 
$\partial_{\mathrm{ext}}Y$.

The spectra of the operators $\mathcal{A}\e^N$ and $\mathcal{A}^D_{\eps}$ are 
purely discrete. We denote by $\left\{\lambda_k^N 
(\eps)\right\}_{k\in\mathbb{N}}$ (respectively, 
$\left\{\lambda_k^D(\eps)\right\}_{k\in\mathbb{N}}$) the sequence of eigenvalues 
of $\mathcal{A}^N_{\eps}$ (respectively, of $\mathcal{A}^D_{\eps}$) arranged in 
the ascending order and repeated according to their multiplicity.

Using the min-max principle and the enclosures
$$\mathrm{dom}(a\e^N) \supset \mathrm{dom}(a\e^\theta)\supset 
\mathrm{dom}(a\e^D)$$
we obtain that
\begin{gather}\label{enclosure}
\forall k\in \mathbb{N},\ \forall\theta\in\mathbb{T}^n:\quad
\lambda_k^N(\eps) \leq \lambda_k^\theta(\eps) \leq
\lambda_k^D(\eps).
\end{gather}

Finally, we present the result of B.~Simon \cite[Theorem 4.1]{S78}, 
which will be widely used 
during the proof. In order to simplify its presentation we introduce an auxiliary definition. 

\begin{definition}\label{def_res}
Let $a$ be a symmetric, closed and positive sesquilinear form in a Hilbert space $H$ with a domain 
$\mathrm{dom}(a)$, which is not necessary dense in $H$. Let $\mathcal{A}$ be a positive 
self-adjoint operator acting in the subspace $\overline{\mathrm{dom}(a)}$ of 
${H}$ and associated with the form $a$.
Then the operator $R$ defined by the formula
$$R=\begin{cases}(\mathcal{A}+I)^{-1}&\text{on 
}\overline{\mathrm{dom}(a)},\ I\text{ is the identity operator},\\
0&\text{on }{H}\ominus \overline{\mathrm{dom}(a)}
\end{cases}$$
is said to be \textit{the generalized resolvent corresponding to the form $a$}. 
\end{definition}

\begin{theorem}[B.~Simon \cite{S78}]
\label{thS}
Let 
$\{a_\eps\}_{\eps>0}$ be a family of closed positive symmetric sesquilinear forms in a Hilbert space 
${H}$, by $\{R\e\}_{\eps>0}$ we denote the corresponding family of generalized 
resolvents. Suppose that $a_\eps$ increases monotonically as $\eps$ 
decreases, i.e.
$$\text{if }\eps_1\geq\eps_2\text{\quad then\quad }\mathrm{dom}(a_{\eps_1})\supset 
\mathrm{dom}(a_{\eps_2})\text{ and }{a}_{\eps_1}[u,u]\leq {a}_{\eps_2}[u,u],\ \forall 
u\in \mathrm{dom}(a_{\eps_2}).$$
Then the positive symmetric sesquilinear form ${a}_0$ defined by
$$\mathrm{dom}(a_0):=\left\{u\in\bigcap\limits_{\eps>0}\mathrm{dom}(a_\eps):\ 
\sup\limits_{\eps>0}a_\eps[u,u]<\infty\right\},\quad 
a_0[u,v]=\liml_{\eps\to 0}a_\eps[u,v]$$
is closed, and moreover 
\begin{gather*}
\forall u\in{H}: R_\eps u \to R_0 u\text{ as }\eps\to 0,
\end{gather*}
where $R_0$ is the generalized resolvent corresponding to the form 
$a_0$. 
\end{theorem}
\smallskip

With these preliminaries we are able to give a short scheme of the proof of 
Theorem \ref{th1}. 
In view of \eqref{repres1}-\eqref{enclosure} the left end (respectively, the 
right end) of the $k$-th spectral band $L_k(\eps)$ is situated between the 
$k$-th Neumann eigenvalue $\lambda_{k}^N(\eps)$ and  the $k$-th periodic 
eigenvalue $\lambda_{k}^\theta(\eps)$, $\theta=(1,\dots,1)$ (respectively, 
between  the $k$-th antiperiodic eigenvalue $\lambda_{k}^\theta(\eps)$, 
$\theta=-(1,\dots,1)$ and the $k$-th Dirichlet eigenvalue 
$\lambda_{k}^D(\eps)$). Our main task is to prove that they both converge to 
$b_{k-1}$ as $k=2,\dots,m+1$ and converge to infinity as $k>m+1$ (respectively, 
converge to $a_{k}$ as $k=1,\dots,m$ and converge to infinity as $k>m$). These 
results taken together constitute the claim of Theorem~\ref{th1}.
Our analysis will be based on Simon's theorem formulated above. 

We notice that the band ends need not in general coincide with the corresponding periodic/antiperiodic eigenvalues, even in case $n=1$ (cf. \cite{HKSW07, EKW10}). What matters is that we can enclose them between two values which converge to the same limit as $\eps\to 0$.

\subsection{\label{subsec22}Asymptotic behaviour of Neumann and periodic 
eigenvalues}

In this subsection we study the behaviour as $\eps\to 0$ of the 
eigenvalues of the operators $\mathcal{A}^N\e$ and $\mathcal{A}^\theta\e$, $\theta=(1,1,\dots,1)$. 

Obviously, $\lambda_1^N(\eps)=0$. 
For the subsequent eigenvalues we will prove the following lemma.

\begin{lemma}\label{lm1}
One has 
\begin{gather*}
\begin{array}{ll}
\liml_{\eps\to 0}\lambda_k^N(\eps)=b_{k-1},&k=2,\dots,m+1,\\
\liml_{\eps\to 0}\lambda_k^N(\eps)=\infty,&k\geq m+2.
\end{array}
\end{gather*} 
\end{lemma}

\begin{proof}

The family of forms $\left\{a\e^N\right\}_{\eps}$ 
increases monotonically as $\eps\to 0$ and we may apply Theorem \ref{thS}. 
Namely,
let us introduce the limit form $a_0^N$,
\begin{gather*}
\mathrm{dom}(a_0^N):=\left\{u\in\mathcal{H}(Y):\ 
\sup\limits_{\eps>0}a^N\e[u,u]<\infty\right\},\quad 
a_0^N[u,v]=\liml_{\eps\to 0}a^N\e[u,v].
\end{gather*}
Evidently $\mathrm{dom}(a_0^N)$
consists of piecewise constant functions, which are continuous in $\overset{\circ}{Y_j}$ for each $j=0,\dots,m$ (this last one follows from \eqref{decomp}-\eqref{conditions} and the definition of $\mathcal{H}(\Gamma)$). Thus $\mathrm{dom}(a_0^N)$ is an $(m+1)$-dimensional subspace of $L_2(Y)$ consisting 
of 
functions $u$ of the form
\begin{gather}
\label{uN}
u(x)=\suml_{j=0}^m \mathbf{u}_j\chi_j(x),\text{ where 
}\mathbf{u}_j\text{ are 
constants, }\chi_j\text{ are the indicator functions of }Y_j
\end{gather}
and, clearly,
$$a_0^N[ u,v ]= \suml_{j=1}^m  q_j N_j \left(\mathbf{u}_0-\mathbf{u}_j\right)\overline{\left(\mathbf{v}_0-\mathbf{v}_j\right)}.$$
We denote by $\mathcal{A}^N_0$ a self-adjoint operator acting in 
$\overline{\mathrm{dom}(a_0^N)}=\mathrm{dom}(a_0^N)$ and associated with the 
form
$a_0^N$. It is straightforward to check that it 
acts as follows:
\begin{gather}\label{A0N}
\mathcal{A}^N_0 u = \left(\suml_{k=1}^m q_k N_k l_0^{-1} (\mathbf{u}_0 
-\mathbf{u}_k)\right)   \chi_0(x) +\suml_{j=1}^m q_j N_j 
l_j^{-1}(\mathbf{u}_j -\mathbf{u}_0)\chi_j(x).
\end{gather}
The operator $\mathcal{A}^N_0$ can be regarded as a Hermitian operator in 
$\mathbb{C}^{m+1}$ equipped with the scalar product 
$(x,y)_{\mathbb{C}^{m+1}}=\suml_{j=0}^{m} l_j x_j\overline{y_j}$.
We denote by 
$$0\leq \lambda^N_1(0)\leq \lambda^N_2(0)\leq\dots\leq \lambda^N_{m+1}(0)$$
its eigenvalues arranged in the ascending order and 
repeated according to their multiplicity. 
It is easy to see that
\begin{gather}\label{lambda0}
\lambda^N_1(0)=0.
\end{gather}
Later we will prove 
\begin{gather}\label{lastN}
\lambda_k^N(0)=b_{k-1},\ k=2,\dots,m+1.
\end{gather}

We denote by $R^{N}_0:L_2(Y)\to L_2(Y)$ the generalized resolvent 
corresponding to the form $a^N_0$. Its spectrum is a union 
of eigenvalues 
\begin{gather}\label{R_spec}
\mu^N_k(0)=(\lambda^N_k(0) + 1)^{-1},\ k=1,\dots,m+1
\end{gather}
and the point $\mu=0$, which is an eigenvalue of infinity 
multiplicity.

Now, applying Theorem \ref{thS} we conclude that
\begin{gather}\label{strong}
\forall u\in L_2(Y):\ (\mathcal{A}^N\e+ I)^{-1}u\to R^N_0 u\text{ as }\eps\to 0. 
\end{gather}
Moreover, since the operators $(\mathcal{A}^N\e+ I)^{-1}$ and $R^N_0$ are compact and $(\mathcal{A}^N_{\eps_1}+ I)^{-1}\geq (\mathcal{A}^N_{\eps_2}+ I)^{-1}\geq 0$
as $\eps_1\geq\eps_2$ then by virtue of the result of T.~Kato \cite[Theorem VIII-3.5]{K66} 
\eqref{strong} can be improved to the norm convergence  
\begin{gather*}
\|(\mathcal{A}^N\e+ I)^{-1}-R^N_0\|\to 0\text{ as 
}\eps\to 
0,
\end{gather*}
whence, using the classical perturbation theory, we obtain the convergence of spectra, namely
\begin{gather}\label{conv1}
\liml_{\eps\to 0}(\lambda^N_k(\eps)+1)^{-1} = \mu^N_k(0)\text{ as } 
k=1,\dots,m+1,\quad
\liml_{\eps\to 0}(\lambda^N_k(\eps)+1)^{-1} = 0\text{ as }k\geq m+2.
\end{gather}
Taking into account \eqref{R_spec} we obtain from \eqref{conv1}:
$$
\liml_{\eps\to 0}\lambda_k^N(\eps)=\lambda_k^N(0)\text{ as }k=1,\dots,m+1,\quad
\liml_{\eps\to 0}\lambda_k^N(\eps)=\infty\text{ as }k\geq m+2.
$$
Thus, to complete the proof of Lemma \ref{lm1} it remains to prove 
\eqref{lastN}.

In view of \eqref{A0N} $\lambda_k^N(0)$, $k=1,\dots,m+1$ are the eigenvalues of 
the  $(m+1)\times(m+1)$ matrix
\begin{gather*}
A=\left(\begin{matrix}\suml_{j=1}^m q_j N_{j} l_0^{-1}&-q_1 N_{1} 
l_0^{-1}&-q_2N_{2} l_0^{-1}&\dots&-q_m l_0^{-1}\\-q_1 N_{1} l_1^{-1}&q_1 N_{1} 
l_1^{-1}&0&\dots&0\\
-q_2 N_{2} l_2^{-1}&0&q_2 N_{2} l_2^{-1}&\dots&0\\
\vdots&\vdots&\vdots&\ddots&\vdots\\
-q_m N_{m} l_m^{-1}&0&0&\dots&q_m N_{m} l_m^{-1}
\end{matrix}\right)
\end{gather*}
They are the roots of the characteristic  equation
$$\mathrm{det}(A-\lambda I)=0.$$

We denote by $M(i_1,i_2,\dots,i_k)$ the minor of the matrix
$A$ staying on the intersection of $i_1$-th,
$i_2$-th,$\dots,i_k$-th rows
 and the columns with the same numbers. One has
the following formula (see, e.g., \cite[\S 2.13.2]{MM64}):
\begin{gather}\label{MM}
\mathrm{det}(A-\lambda I)=\suml_{k=0}^{m+1}\lambda^{m+1-k}(-1)^{m+1-k} E_k ,
\end{gather}
where 
\begin{gather}\label{EEE}
E_0=1,\quad E_k=\sum_{1\leq
i_1<i_2<\dots < i_k\leq m+1}M(i_1,i_2,\dots,i_k)\text{ as }k\geq 1.
\end{gather}

It is clear that $E_{m+1}=\mathrm{det}(A)=0$ since the sum of all columns of 
$A$ is zero. 
For $2\leq k\leq m$ we represent $E_k$ as a sum of two terms:
\begin{gather}\label{minor1}
E_k=
\sum_{1\leq
i_1<i_2<\dots < i_k\leq m}M(i_1+1,i_2+1,\dots,i_k+1) + 
\sum_{1\leq
i_2<\dots < i_k\leq m}M(1,i_2+1,\dots,i_k+1).
\end{gather}
One has (below ${1\leq
i_1<i_2<\dots < i_k\leq m}$):
\begin{gather}\label{minor2}
M(i_1+1,i_2+1,\dots,i_k+1)=
\mathrm{det}\left(\begin{matrix}q_{i_1}N_{i_1} l_{i_1}^{-1}&0&\dots&0\\
0&q_{i_2}N_{i_2} l_{i_2}^{-1}&\dots&0\\
\vdots&\vdots&\ddots&\vdots\\
0&0&\dots&q_{i_k}N_{i_k} l_{i_k}^{-1}
\end{matrix}\right)=
\prod_{\a=1}^k {q_{i_\a}N_{i_\a} l^{-1}_{i_\a}}.
\end{gather}
and (below ${1\leq
i_2<\dots < i_k\leq m}$)
\begin{multline}\label{minor3}
M(1,i_2+1,\dots,i_k+1)=
\mathrm{det}\left(\begin{matrix}\suml_{j=1}^m q_j N_j l_0^{-1}&-q_{i_2} N_{i_2} 
l_0^{-1}&-q_{i_3} N_{i_3} l_0^{-1}&\dots&-q_{i_k} N_{i_k} l_0^{-1}\\-q_{i_2} 
N_{i_2} l_{i_2}^{-1}&q_{i_2}N_{i_2} l_{i_2}^{-1}&0&\dots&0\\
-q_{i_3} N_{i_3} l_{i_3}^{-1}&0&q_{i_3} N_{i_3} l_{i_3}^{-1}&\dots&0\\
\vdots&\vdots&\vdots&\ddots&\vdots\\-q_{i_k}N_{i_k} l_{i_k}^{-1}&
0&0&\dots&q_{i_k}N_{i_k} l_{i_k}^{-1}
\end{matrix}\right)\\=
\mathrm{det}\left(\begin{matrix}\suml_{\a=2}^k q_{i_\a}N_{i_\a} 
l_0^{-1}&-q_{i_2} N_{i_\a}l_0^{-1}&-q_{i_3}N_{i_\a} 
l_0^{-1}&\dots&-q_{i_k}N_{i_k} l_0^{-1}\\-q_{i_2}N_{i_2} 
l_{i_2}^{-1}&q_{i_2}N_{i_2} l_{i_2}^{-1}&0&\dots&0\\
-q_{i_3}N_{i_3} l_{i_3}^{-1}&0&q_{i_3}N_{i_3} l_{i_3}^{-1}&\dots&0\\
\vdots&\vdots&\vdots&\ddots&\vdots\\-q_{i_k}N_{i_k} l_{i_k}^{-1}&
0&0&\dots&q_{i_k}N_{i_k} l_{i_k}^{-1}
\end{matrix}\right)\\+
\mathrm{det}\left(\begin{matrix}\suml_{j\notin \{i_2,\dots , i_k\}} q_jN_{j} 
l_0^{-1}&0&0&\dots&0\\-q_{i_2} N_{i_2} l_{i_2}^{-1}&q_{i_2}N_{i_2} 
l_{i_2}^{-1}&0&\dots&0\\
-q_{i_3}N_{i_3} l_{i_3}^{-1}&0&q_{i_3}N_{i_3} l_{i_3}^{-1}&\dots&0\\
\vdots&\vdots&\vdots&\ddots&\vdots\\-q_{i_k}N_{i_k} l_{i_k}^{-1}&
0&0&\dots&q_{i_k}N_{i_k} l_{i_k}^{-1}
\end{matrix}\right).
\end{multline}
The first determinant in the right-hand-side of \eqref{minor3} is equal to zero 
since the sum all 
columns of the corresponding matrix  is equal to zero. As a result we obtain:
\begin{gather}\label{minor4}
M(1,i_2+1,\dots,i_k+1)=\left(\suml_{j\notin \{i_2,\dots ,i_k\}} q_j N_{j} 
l_0^{-1}\right)\left(\prod_{\a=2}^k {q_{i_\a}N_{i_\a} l^{-1}_{i_\a}}\right).
\end{gather}
Via a simple algebraic calculations it is not hard to get from \eqref{minor4} that
\begin{gather}\label{minor4+}
\suml_{1\leq
i_2<\dots < i_k\leq m}M(1,i_2+1,\dots,i_k+1)=l_0^{-1}\sum_{1\leq
i_1<i_2<\dots < i_k\leq m}\left(\left(\prod_{\a=1}^k  q_{i_\a}N_{i_\a}l^{-1}_{i_\a} 
\right)\left(\suml_{\a=1}^k l_{i_\a}\right)\right).
\end{gather}

Combining \eqref{minor1}, \eqref{minor2}, \eqref{minor4+} and taking into account the definition of $a_j$ one arrives 
at
\begin{gather}\label{minor5}
E_k=\sum_{1\leq
i_1<i_2<\dots < i_k\leq m}\left(\left(\prod_{\a=1}^k  a_{i_\a} 
\right)\left(1+l_0^{-1}\suml_{\a=1}^k l_{i_\a}\right)\right).
\end{gather}
We have proved \eqref{minor5} for $2\leq k\leq m$. For $k=1$ it holds as well:
$$E_1\overset{\eqref{EEE}}=\mathrm{tr}A=\suml_{i=1}^m q_i N_i l_0^{-1}+\suml_{i=1}^m q_i N_i l_i^{-1}=
\suml_{i=1}^m a_i \left(1+l_0^{-1}l_{i}\right).$$

Now let us study the function $\mathcal{F}(\lambda)$ staying in the 
right-hand-side of equation \eqref{mu_eq}. 
One has 
\begin{gather}
\mathcal{F}(\lambda)= {1\over \prod\limits_{j=1}^m (a_j-\lambda)} \left( 
\prod\limits_{j=1}^m (a_j-\lambda) +l_0^{-1}\suml_{i=1}^m\left( a_i l_i \prod\limits_{j\not= 
i}(a_j-\lambda) \right)\right).
\end{gather}
Grouping the terms with the same exponents of $\lambda$ one can 
easily obtain:
\begin{gather}
\mathcal{F}(\lambda)= {1\over\prod\limits_{j=1}^m (a_j-\lambda)}\suml_{k=0}^m 
\lambda^{m-k}\left((-1)^{m-k} \sum_{1\leq
i_1<i_2<\dots < i_k\leq m}\left(\left(\prod_{\a=1}^k  a_{i_\a} 
\right)\left(1+l_0^{-1}\suml_{\a=1}^k l_{i_\a}\right)\right)\right)
\end{gather}
or, using \eqref{MM}, \eqref{minor5} and taking into account that 
$E_{m+1}=0$, we obtain:
\begin{multline}
\mathcal{F}(\lambda)= {1\over \prod\limits_{j=1}^m (a_j-\lambda)}\suml_{k=0}^m 
\lambda^{m-k}(-1)^{m-k} E_k= {1\over -\lambda  \prod \limits_{j=1}^m (a_j-\lambda)}\suml_{k=0}^m \lambda^{m+1-k}(-1)^{m+1-k} 
E_k\\={1\over -\lambda \prod\limits_{j=1}^m (a_j-\lambda)}\suml_{k=0}^{m+1} 
\lambda^{m+1-k}(-1)^{m+1-k} E_k= {1\over -\lambda \prod\limits_{j=1}^m (a_j-\lambda)} 
\mathrm{det}(A-\lambda I),
\end{multline}
whence, taking into account \eqref{abab} and \eqref{lambda0}, we easily obtain 
\eqref{lastN}. 
The lemma is proved.
\end{proof}

The same asymptotics are valid for the eigenvalues of the operator 
$\mathcal{A}\e^\theta$
as $\theta= (1,1,\dots,1)$.
\begin{lemma}\label{lm1+}
One has 
\begin{gather*}
\begin{array}{ll}
\liml_{\eps\to 0}\lambda_k^\theta(\eps)=b_{k-1},&k=2,\dots,m+1,\\
\liml_{\eps\to 0}\lambda_k^\theta(\eps)=\infty,&k\geq m+2
\end{array}
\end{gather*}
provided $\theta= (1,1,\dots,1)$. 
\end{lemma}

\begin{proof}
It is easy to see that functions $u$ of the form \eqref{uN}
belong to 
$\mathrm{dom}(a\e^{\theta})$ provided $\theta= (1,1,\dots,1)$, whence, 
evidently,
the limit form 
$a_0^\theta$ coincides 
with the form $a_0^N$.
In the rest the proof repeats word-by-word the proof of Lemma \ref{lm1}.
\end{proof}

\subsection{\label{subsec23}Asymptotic behaviour of Dirichlet and 
$\theta$-periodic eigenvalues ($\theta\not=(1,1,\dots,1)$)}

In this subsection we study the behaviour as $\eps\to 0$ of the 
eigenvalues of the operators $\mathcal{A}^D\e$ and $\mathcal{A}^\theta\e$, $\theta\not=(1,1,\dots,1)$. 
 
\begin{lemma}\label{lm2}
One has 
\begin{gather*}
\begin{array}{ll}
\liml_{\eps\to 0}\lambda_k^D(\eps)=a_{k},&k=1,\dots,m,\\
\liml_{\eps\to 0}\lambda_k^D(\eps)=\infty,&k\geq m+1.
\end{array}
\end{gather*} 
\end{lemma}

\begin{proof}
For the proof we employ the same method as in the proof of Lemma \ref{lm1}.
Namely, we again introduce the limit form $a\e^D[u,u]$ by 
\begin{gather*}
\mathrm{dom}(a_0^D):=\left\{u\in\mathcal{H}(Y),\ u|_{\partial_{\mathrm{ext}} 
Y}=0:\ 
\sup\limits_{\eps>0}a^D\e[u,u]<\infty\right\},\quad 
a_0^D[u,v]=\liml_{\eps\to 0}a^D\e[u,v].
\end{gather*}
It is clear that $\mathrm{dom}(a_0^D)$
consists of piecewise constant functions, which are continuous in 
$\overset{\circ}{Y_j}$ for each $j=1,\dots,m$ and equal to zero in 
$\overset{\circ}{Y_0}$.
Thus 
$\mathrm{dom}(a_0^D)$ is an $m$-dimensional subspace of $L_2(Y)$ consisting of 
functions $u$ of the form
$$u(x)=\suml_{j=1}^m \mathbf{u}_j\chi_j(x),\text{ where }\mathbf{u}_j\text{ are 
constants, }\ \chi_j\text{ is an indicator function of }Y_j$$
and
$$a_0^D[ u,v ]= \suml_{j=1}^m  q_j N_j \mathbf{u}_j\overline{\mathbf{v}_j}.$$
We denote by $\mathcal{A}^D_0$ a bounded and self-adjoint operator acting in 
$\overline{\mathrm{dom}(a_0^D)}=\mathrm{dom}(a_0^D)$ and associated with the 
form
$a_0^D$. It acts as follows:
\begin{gather}\label{A0D}
\mathcal{A}^D_0 u = \suml_{j=1}^m q_j N_j l_j^{-1}\mathbf{u}_j  
\chi_j(x).
\end{gather}

Repeating word-by-word the arguments of the proof of Lemma \ref{lm1} we 
conclude 
that
\begin{gather*}
\begin{array}{ll}
\liml_{\eps\to 0}\lambda_k^D(\eps)=\lambda_{k}^D(0),&k=1,\dots,m,\\
\liml_{\eps\to 0}\lambda_k^D(\eps)=\infty,&k\geq m+1,
\end{array}
\end{gather*} 
where $\lambda_k^D(0)$ is the $k$-th eigenvalue of the operator 
$\mathcal{A}_0^D$. 
It follows from \eqref{A0D} that
$$\lambda_k^D(0)={q_k N_k l_k^{-1}}=a_k.$$
The lemma is proved.
\end{proof}

The same asymptotics are valid for the eigenvalues of the operator 
$\mathcal{A}\e^\theta$
as $\theta\not= (1,1,\dots,1)$.

\begin{lemma}\label{lm2+}
One has 
\begin{gather*}
\begin{array}{ll}
\liml_{\eps\to 0}\lambda_k^\theta(\eps)=a_{k},&k=1,\dots,m,\\
\liml_{\eps\to 0}\lambda_k^\theta(\eps)=\infty,&k\geq m+1.
\end{array}
\end{gather*}
provided $\theta\not= (1,1,\dots,1)$. 
\end{lemma}

\begin{proof}
The definitional domain of the form $a\e^{\theta}$ consists 
of functions $u$ having the form \eqref{uN}
and belonging to 
$\mathrm{dom}(a\e^{\theta})$. 
It is easy to see that if $\theta\not= (1,1,\dots,1)$ then
$\mathbf{u}_0=0$ (otherwise $u\notin \mathrm{dom}(a\e^{\theta})$).
Thus
the limit form 
$a_0^\theta$ coincides 
with the form $a_0^D$ provided $\theta\not= (1,1,\dots,1)$.
In the rest the proof repeats word-by-word the proof of Lemma \ref{lm2}.
\end{proof}

\subsection{\label{subsec24}End of the proof of Theorem \ref{th1}}

Due to \eqref{repres1}-\eqref{repres1+} one has
\begin{gather}\label{sp}
\sigma(\mathcal{A}\e)=\cupl_{k=1}^\infty [\lambda_k^-(\eps),\lambda_k^+(\eps)]
\end{gather}
with the compact intervals are $[\lambda_k^-(\eps),\lambda_k^+(\eps)]$ 
defined as follows:
\begin{gather}\label{interval}
[\lambda_k^-(\eps),\lambda_k^+(\eps)]= \cupl_{\theta\in \mathbb{T}^n}
\left\{\lambda_k^{\theta}(\eps)\right\}.
\end{gather}

We denote $\theta_1:=(1,1,\dots,1)$, $\theta_{-1}:=-\theta_1$.
Using \eqref{enclosure} and \eqref{interval} we conclude that
\begin{gather}\label{double1}
\lambda_k^{N}(\eps)\leq \lambda_k^-(\eps)\leq
\lambda_k^{\theta_1}(\eps),\\\label{double2}
\lambda_k^{\theta_{-1}}(\eps)\leq \lambda_k^+(\eps)\leq
\lambda_k^{D}(\eps).
\end{gather}
The left and right-hand-sides of
\eqref{double1} are equal to zero as $k=1$. In view of Lemmata \ref{lm1}, 
\ref{lm1+} if $k=2,\dots,m+1$ 
they both converge to $b_{k-1}$ as $\eps\to 0$, while if $k\ge
m+2$ they converge to infinity. Hence
\begin{gather}\label{a-}
\lambda_1^-(\eps)=0,\quad \liml_{\eps\to 0}\lambda_k^-(\eps)=b_{k-1}\text{
if }2\leq k\leq m+1,\quad \liml_{\eps\to
0}\lambda_k^-(\eps)=\infty\text{ if }k\ge m+2.
\end{gather}
Similarly in view Lemmata \ref{lm2}, \ref{lm2+} we obtain
\begin{gather}\label{a+}
\liml_{\eps\to 0}\lambda_k^+(\eps)=a_k\text{ if }1\leq k\leq
m,\quad  \liml_{\eps\to 0}\lambda_k^+(\eps)=\infty\text{ if }k\ge m+1.
\end{gather}
Then \eqref{th1_f1}-\eqref{th1_f2} follow directly from
\eqref{sp}, \eqref{a-}, \eqref{a+}. Theorem \ref{th1} is proved.

\section{\label{sec3}Periodic quantum graphs with asymptotically predefined 
spectral gaps}

In this section we will show that under a suitable choice of the graph $\Gamma$ and the coupling constants $q_j$ the limit intervals $(a_j,b_j)$ coincide with predefined ones.

Let $\Gamma$ be a $\mathbb{Z}^n$-periodic graph with a periodic cell $Y$ 
admitting decomposition \eqref{decomp}-\eqref{conditions}. Recall that the 
notation 
$l_j$ stays for the total length of all edges belonging to the set
$Y_j$ ($j=0,\dots,m$), by $N_j$ we denote the number of points belonging to 
the set $\mathcal{V}_j$ ($j=1,\dots,m$) -- see Section \ref{sec1}, where these 
notations are introduced. Also in the same way as before we introduce the 
numbers $a_j$ and $b_j$ ($j=1,\dots,m$).

\begin{theorem}
\label{th2}
Let $L>0$ be an arbitrarily large number and let $(\a_j,\b_j)$
($j={1,\dots,m},\ m\in \mathbb{N} $) be arbitrary intervals
satisfying
\begin{gather}\label{intervals}
0<\a_1,\quad \a_j<\b_{j}< \a_{j+1},\ j=\overline{1,m-1},\quad
\a_m<\b_{m}<L.
\end{gather} 
Suppose that the numbers $l_j$, $j=0,\dots,m$, satisfy 
\begin{gather}\label{exact_formula}
l_j=l_0\ds{{\b_j-\a_j\over\a_j}\prod\limits_{i=\overline{1,m}|i\not=
j}\ds\left({\b_i-\a_j\over \a_i-\a_j}\right)}.
\end{gather}

Then one has
\begin{gather}\label{a=a-b=b}
a_j=\a_j,\ b_j=\b_j,\quad j=1,\dots,m
\end{gather}
provided
\begin{gather}\label{q=}
q_j= {\a_j {l}_j \over N_j },\ j=1,\dots,m.
\end{gather}

\end{theorem}

\begin{remark}
Since 
the intervals $(\a_j,\b_j)$ satisfy \eqref{intervals}
then
\begin{gather*}
\forall j:\ \b_j>\a_j,\quad \forall i\not= j:\
\mathrm{sign}(\b_i-\a_j)=\mathrm{sign}(\a_i-\a_j)\not= 0
\end{gather*}
and therefore  
the quantity staying in the right-hand-side of \eqref{exact_formula} is positive.
\end{remark}

\begin{remark}
Results, similar to Theorem \ref{th2} (i.e., construction of periodic operators with gaps that are close to given intervals), were obtained by one of the authors in \cite{Kh12} for  Laplace-Beltrami  operators on periodic Riemannian manifolds, in \cite{Kh13} for periodic elliptic  divergence type operators in $\mathbb{R}^n$, and in \cite{Kh14} for Neumann Laplacians in periodic domains in $\mathbb{R}^n$. 
\end{remark}

\begin{proof}
Plugging \eqref{q=} into \eqref{aj} we obtain
the first equality of \eqref{a=a-b=b}.

Recall that the numbers $b_j$ are the roots of the equation  \eqref{mu_eq} written in the ascending order. Therefore, in order to prove the second equality in \eqref{a=a-b=b} one has to show that
\begin{gather}\label{system}
\forall i=1,\dots,m:\quad 1+\suml_{j=1}^m{\a_j l_j\over l_0(\a_j-\b_i)}=0.
\end{gather}
Let us consider (\ref{system}) as the linear algebraic system of
$m$ equations with unknowns $l_j$, $j=1,\dots,m$. 
It was proved in \cite[Lemma 4.1]{Kh12} that this system has the unique solution defined by formula \eqref{exact_formula}. This implies the   second equality in \eqref{a=a-b=b}. Theorem \ref{th2} is proved.
\end{proof}

It is easy to construct the graph $\Gamma\subset\mathbb{R}^d$ satisfying 
\eqref{decomp}-\eqref{conditions} and \eqref{exact_formula}. For example, one 
can proceed as follows.
Let $\Gamma_0$ be an arbitrary $\mathbb{Z}^n$-periodic metric graph, 
$e_1,\dots,e_n$ be vectors producing an action of $\mathbb{Z}^n$ on $\Gamma_0$ 
(i.e., \eqref{periodicity} holds). 
We denote by $Y_0$ its period cell. Let $v_1,\dots,v_m$ be arbitrary points belonging to $\overset{\circ}{Y_0}$.
Let $Y_j$, $j=1,\dots,m$ be arbitrary compact graphs satisfying
$Y_i\cap Y_j=\varnothing$ as $i\not= j$ and $Y_j\cap  \Gamma_0= \{v_j\}$.
We denote $$Y_{ij}=Y_j+\suml_{k=1}^n i_k e_k,\ i=(i_1,\dots,i_n)\in\mathbb{Z}^n$$ 
and, finally,
$$\Gamma= 
\Gamma_0\cupl\left(\cupl_{i\in\mathbb{Z}^n}\cupl_{j=1}^m Y_{ij}\right).$$
The graph $\Gamma$ is presented on Figure \ref{fig2} (here the graph in $\mathbb{Z}$-periodic, $m=2$).
The set
$$Y:=\cupl_{j=0}^m Y_j$$
is a period cell of $\Gamma$. It is easy to see that the sets $Y_j$ satisfy 
conditions \eqref{conditions}. Obviously, they can be chosen in such a way that \eqref{exact_formula} holds -- the simplest way is to take $$Y_j=\left\{\text{single edge of the length }l_j\text{ defined by formula \eqref{exact_formula}}\right\}.$$

\section*{Acknowledgements}

The authors express their gratitude to Prof. Pavel Exner for fruitful discussion on the results.
The work of D.B. is supported by Czech Science Foundation (GACR), the project 14-02476S ``Variations, geometry and physics'', by the project ``Support of Research in the
Moravian-Silesian Region 2013'' and by the University of Ostrava.
A.K. is grateful for hospitality extended to him during several visits to the Department of Mathematics of University of Ostrava where a part of this work was done.

\end{document}